\documentclass[12pt,reqno]{amsart}

\usepackage{amsfonts,amsmath,amsthm}
\usepackage{amssymb,epsfig,color}
\usepackage{enumerate} 
\usepackage{mathrsfs}
\usepackage{color} %color
\definecolor{vert}{rgb}{0,0.6,0}

\usepackage[text={425pt,650pt},centering]{geometry}
\usepackage{graphicx}
\usepackage{epsfig}
\usepackage{tikz}
\usepackage{caption}
\usepackage{color} %color
\definecolor{vert}{rgb}{0,0.6,0}

\numberwithin{figure}{section}

\theoremstyle{plain}
\newtheorem{thm}{Theorem}[section]

\newtheorem{defn}{Definition}

\newtheorem{lem}[thm]{Lemma}

\theoremstyle{remark}
\newtheorem{rem}{\bf{Remark}}
\numberwithin{equation}{section}

\newcommand{\N}{\mathbb{N}}

\newcommand{\R}{\mathbb{R}}

\newcommand{\T}{\mathbb{T}}
\newcommand{\Z}{\mathbb{Z}}

\newcommand{\cA}{\mathcal{A}}

\newcommand{\cE}{\mathcal{E}}
\newcommand{\cR}{\mathcal{R}}

\newcommand{\sC}{\mathscr{C}}

\newcommand{\BUC}{{\rm BUC\,}}

\newcommand{\Lip}{{\rm Lip\,}}

\newcommand{\al}{\alpha}
\newcommand{\gam}{\gamma}
\newcommand{\del}{\delta}
\newcommand{\ep}{\varepsilon}

\newcommand{\lam}{\lambda}

\newcommand{\ol}{\overline}

\newcommand{\co}{{\rm co}\,}

\begin{document}

\title[Effective fronts of polytope shapes]
{Effective fronts of polytope shapes}

\author[W. JING, H. V. TRAN, Y. YU]
{Wenjia Jing, Hung V. Tran, Yifeng Yu}

\thanks{
The work of WJ is partially supported by the NSFC Grants No.\,11701314 and No.\,11871300.
The work of HT is partially supported by NSF grant DMS-1664424 and NSF CAREER grant DMS-1843320.
%The work of YY is partially supported by NSF CAREER award \#1151919.
}

\address[W. Jing]
{
Yau Mathematical Sciences Center,
Tsinghua University, No.1 Tsinghua Yuan,
Beijing 100084, China }
\email{wjjing@tsinghua.edu.cn}

\address[H. V. Tran]
{
Department of Mathematics, 
University of Wisconsin Madison, 480 Lincoln  Drive, Madison, WI 53706, USA}
\email{hung@math.wisc.edu}

\address[Y. Yu]
{
Department of Mathematics, 
University of California, Irvine,
410G Rowland Hall, Irvine, CA 92697, USA}
\email{yyu1@math.uci.edu}

\date{\today}

\maketitle

\begin{abstract}
We study the periodic homogenization of first order front propagations.
Based on PDE methods, we  provide a simple proof that  for $n \geq 3$, the class of centrally symmetric polytopes with rational coordinates and nonempty interior is admissible as effective fronts, which was also established in  \cite{BB2006, J2009} in the form of stable norms as an extension of Hedlund's classical result \cite{Hed}.  
Besides, we  obtain the optimal convergence rate of the homogenization problem for this class. 
\end{abstract}

%%%%%%%%%%%%%%%%%%%%%%%%%%%%%%%%%%%%%%%%%%%

\section{Introduction}

In this paper, we study a couple of fine questions related to the homogenization of the following Hamilton-Jacobi equation:
\begin{equation}\label{HJ-ep}
\begin{cases}
u^\ep_t + a\left(\frac{x}{\ep}\right) |Du^\ep| = 0 \quad &\text{ in } \R^n \times (0,\infty),\\
u^\ep(x,0) = g(x) \quad &\text{ on } \R^n.
\end{cases}
\end{equation}
Here, $g\in \BUC(\R^n) \cap \Lip(\R^n)$ is the initial data, where $\BUC(\R^n)$ is the space of bounded, uniformly continuous functions on $\R^n$. The coefficient $a: \R^n \to \R$ determines the normal velocity in the underlying front propagation model; the small number $\ep \in (0,1)$ is the spatial scale of variations in $a$. This problem hence models front propagations in oscillatory environment. We assume throughout the paper that $a$ is a continuous,  $\Z^n$-periodic and non-constant positive function.

The qualitative homogenization of \eqref{HJ-ep} fits in the classical and standard framework (see \cite{LPV, Ev1} for example). %As far as $a:\R^n \to \R$ is concerned, we assume throughout the paper that $a$ is continuous,  $\Z^n$-periodic, and there exist $\al,\beta>0$ such that
%\[
%\al \leq a(y) \leq \beta \quad \text{ for all } y \in \R^n.
%\]
Here is a brief review of the result. As $\ep \to 0$, the solution $u^\ep$ of the above problem converges, locally uniformly in $\R^n \times [0, \infty)$, to the solution of the effective or homogenized problem:
\begin{equation}\label{HJ}
\begin{cases}
u_t + \ol{H}(Du) = 0 \quad &\text{ in } \R^n \times (0,\infty),\\
u(x,0) = g(x) \quad &\text{ on } \R^n.
\end{cases}
\end{equation}
Here, $\ol{H}$ is the so-called effective Hamiltonian. For each $p\in \R^n$, $\ol{H}(p)$ is the unique real  number such that the following equation admits a continuous viscosity solution
\[
 {\rm (E)}_p \qquad a(y) |p +Dv_p(y)| = \ol{H}(p) \quad \text{ in } \T^n.
\] 
This is the well-known cell problem. 

An important feature of problem \eqref{HJ-ep}, due to its special structure, is that the effective Hamiltonian $\ol{H}$ is also intrinsically determined by a shape theorem which says the large time average of the reachable set  from the origin, determined by the environment function $a$, converges to a centrosymmetric convex set. It turns out that $\ol{H}$ is the support function of this convex set; see \eqref{eq:HD} below.

In this paper, we aim to study an inverse shape theorem: what class of convex sets are admissible as the limiting shapes of the averaged reachable sets for some environment function $a$?  We show that for $n \geq 3$, all centrally symmetric polytopes with rational coordinates and nonempty interior are admissible.
The second objective of this paper is to obtain optimal convergence rate of the homogenization problem, and we show this is possible in general when $\ol{H}$ is determined by centrally symmetric polytopes type limiting shapes.
%Besides, we also obtain the optimal convergence rate of the homogenization problem for this class. 

\smallskip

Let us now present the reachable set framework and our main results.

\subsection{Admissible paths and reachable sets}

For $t>0$, let $\cA_{0,t}$ be the set of admissible paths defined as follows
\begin{equation*}
\cA_{0,t}=\left\{\gam \in W^{1,\infty}([0,t],\R^n) \,:\, |\dot \gam(s)| \leq a(\gam(s)) \quad \text{ for a.e. } s \in [0,t] \right\}.
\end{equation*}
The reachable set $\cR_t(x)$ at time $t>0$ emanating from $x\in \R^n$ at time $0$ is defined as following
\[
\cR_t(x)=\left\{ y\in \R^n\,:\, \text{ there exists $\gam \in \cA_{0,t}$ such that } \gam(0)=x, \gam(t)=y \right\}.
\]
Note that since $a$ is periodic and positive, there exist a lower bound $\alpha >0$ and an upper bound $\beta > \alpha$ such that, for all $x\in \R^n$, 
\begin{equation*}
0< \alpha \le a(x) \le \beta.
\end{equation*}
It is then easy to verify that, for any fixed $x \in \R^n$, the set $\cR_t(x)$ is increasing with respect to $t$. We are interested in the large time average of $\cR_t(x)$, that is, the behavior of
\[
\frac{\cR_t(x)}{t} \qquad \text{ as } t \to \infty.
\]
We also denote by $\cR_t(Y)$ the union of $\cR_t(x)$, $x\in Y$. To describe this behavior, we need the following notions.
Let $\sC$ denote the set of non-empty compact subsets of $\R^n$. 
The Hausdorff distance between $E$ and $F$ in $\sC$ is defined by
\begin{align*}
\rho(E,F) &= \max \left \{ \sup_{x\in F} \inf_{y \in E} |x-y|, \ \sup_{y \in E} \inf_{x \in F} |x-y|\right\}\\
&=\inf \left\{ s>0 \,:\, F \subset E + \ol{B}_s,\, E \subset F + \ol{B}_s \right\}.
\end{align*}
It is well known that $(\sC,\rho)$ is a complete metric space. The existence of the limiting shape below is well known to experts. 
\begin{lem}[A Shape Theorem]\label{lem:R-t-conv}
There exists  a compact and convex set $D \subset \R^n$ such that
\begin{equation}\label{R-t-conv}
\lim_{t \to \infty} \frac{\cR_t(x)}{t} = D \quad \text{ in } \quad (\sC,\rho).
\end{equation}
\end{lem}
If needed, we write $D_a$ instead of $D$ to demonstrate the clear dependence of the limit \eqref{R-t-conv} in $a$. An important property of \eqref{HJ-ep} is, as mentioned before, the effective Hamiltonian $\ol{H}$ of the homogenized problem, which is usually determined by ${\rm (E)}_p$, is the support function of the above limit shape $D$. This means: for $p \in \R^n$,
\begin{equation}\label{eq:HD}
\ol{H}(p) = \sup_{q \in D} \, p\cdot q.
\end{equation}
We then have the following result.

\begin{thm}[Qualitative homogenization result for \eqref{HJ-ep}]\label{thm:hom}
For each $\ep>0$, let $u^\ep$ be the unique viscosity solution to \eqref{HJ-ep}.
Let $u$ be the solution of \eqref{HJ}, with $\ol{H}$ defined by \eqref{eq:HD}. Then, as $\ep \to 0$, $u^\ep$ converges locally uniformly on $\R^n \times [0,\infty)$ to the solution $u$ of \eqref{HJ}.\end{thm}

The qualitative homogenization result for \eqref{HJ-ep} with $\ol{H}$ determined by the cell problem ${\rm(E)}_p$, even in a more general framework, is quite well-known in the literature (see \cite{LPV, Ev1} for example). The result above shows that in the specific case of \eqref{HJ-ep}, $\ol{H}$ can also be characterized by \eqref{eq:HD}. We used both forms of $\ol{H}$ in our analysis.

\subsection{Main results}
We proceed to study deeper properties of $\ol{H}$ and the rate of convergence of $u^\ep$ to $u$. In the sequel, the usual $n$-dimensional flat torus is denoted by $\T^n= \R^n / \Z^n$, and all polytopes are assumed to be centered at the origin. We say a non-zero vector $p \in \R^n$ is rational if  $sp$ has integer coordinates for some $s>0$.

\begin{thm}\label{thm:polytope}
Assume that $n \geq 3$.
Let $P \subset \R^n$ be a centrally symmetric polytope with rational vertices and nonempty interior.
Then, there exists an environment function $a \in C^\infty(\T^n, (0,\infty))$ such that $D_a=P$.
\end{thm}

The set of centrally symmetric polytopes with rational vertices and nonempty interior is important since it is dense in the set of centrally symmetric convex sets with nonempty interior. 
The same result has also been proved   in  \cite{BB2006, J2009} in the equivalent form of stable norms. 
Constructions in these works are all based on Hedlund's original idea \cite{Hed} where a cubic $D$ is constructed. 
The main difference lies in the methods/computations involved in verifying that $D$ is indeed the desired polytope. 
Our proof based on PDE/weak KAM type approaches is much simpler than those delicate  geometric calculations in \cite{Hed,BB2006,J2009}.
This result  belongs to the ongoing project of studies of inverse problems in periodic homogenization of Hamilton-Jacobi equations (see \cite{LTY, JTY, TY}). 
 A very interesting question is whether every  centrally symmetric convex set with nonempty interior is realizable within the class of $a \in C(\T^n,(0,\infty))$.  
 A natural thought is to achieve this by approximations. 
 However, our construction actually shows that $\overline H$ is not stable under convergence of $a$ that is weaker than the uniform convergence (see Remark \ref{rem:unstability}). 
In addition, we would like to point out that, when $n=2$, polytopes are not realizable by $a\in C^1(\Bbb T^2,(0,\infty))$, because it is proved in  \cite{Car} that  $a \in C^1$ yields strictly convex limiting shape $D$ (equivalently, the dual set $\left\{\overline H=1\right\}$ is $C^1$). As long as the continuous class of $a$ is considered,  we believe that the above theorem should hold in two dimensions.  
We will investigate these problems in the future.  

\smallskip

Next is our optimal rate of convergence result in case that the effective front is a centrally symmetric polytope with nonempty interior.

\begin{thm}\label{thm:rate2}
Let $P \subset \R^n$ be a centrally symmetric polytope with nonempty interior.
Let $a \in C(\T^n, (0,\infty))$ be such that $D_a=P$.
Then, there exists a constant $C>0$ depending only on $a$ such that
\[
\|u^\ep -u\|_{L^\infty(\R^n \times [0,\infty))} \leq C \ep.
\]
\end{thm}

The first quantitative result in periodic homogenization of Hamilton-Jacobi equations was obtained in \cite{CDI} with rate $O(\ep^{1/3})$.
Afterwards, optimal rate of convergences $O(\ep)$ were derived in various convex settings in \cite{MTY}.
In particular, if the Hamiltonian $H=H(y,p)$ is homogeneous in the $p$ variable (e.g., $H(y,p)=a(y)|p|$ in our case) in two dimensions,  $O(\ep)$ holds for any Lipschitz continuous initial data $g(x)$ (see  \cite[Theorem 1.2]{MTY}). 
If $n\geq 3$, we believe that extra assumptions on $\overline H$ are necessary in general in order to obtain $O(\ep)$ although an example with a fractional convergence rate is still elusive. 
This theorem is a next development along the line of \cite{MTY} for $n\geq 3$. See Theorem 1.1 in \cite{MTY} for other more ``generic" assumptions.

\begin{rem}  The assumption  in Theorem \ref{thm:polytope} that the polytope $D$ has rational vertices is necessary in our constructive proof. 
It is not clear to us,  for smooth $a$,  whether $D_a$ being a polytope automatically implies that all vertices are rational vectors.  
According to \cite{BIK}, if $D$ is a polytope, no vertex could be an irrational vector (i.e., its coordinates are linearly independent over $\mathbb{Q}$).  
However, there are vectors which are neither rational nor irrational, for example $q=(1,2,\sqrt{3}) \in \R^3$. 
\end{rem}

\subsection*{Outline of the paper}
In Section \ref{sec:Hed}, we give the proof of Theorem \ref{thm:polytope}, which is a generalization of the classical Hedlund example.
In Section \ref{sec:op}, we prove the optimal convergence result (Theorem \ref{thm:rate2}).
In the Appendix, we present the results on large time average of reachable sets and the qualitative homogenization result (proofs of Lemma \ref{lem:R-t-conv} and Theorem \ref{thm:hom}).
Although Lemma \ref{lem:R-t-conv} and Theorem \ref{thm:hom} are classical and standard in the literature, for readers' convenience,  we provide their proofs based on the reachable set framework, which are quantifiable and more in line with the approaches used in this paper.

\subsection*{Notations}
Let $Y$ be the unit cell $[0,1]^n$ and $\tilde Y = - Y= [-1,0]^n$.
For any $x\in \R^n$, denote by $([x],\hat x)$ the unique pair in $\Z^n \times [0,1)^n$ such that $x=[x]+\hat x$.
Denote by $B_r(x), \ol{B}_r(x)$ the open ball, closed ball of center $x$, radius $r>0$ in $\R^n$, respectively. 
We write $B_r=B_r(0), \ol{B}_r= \ol{B}_r(0)$ for short.
For $E,F \subset \R^n$ nonempty and $c\in \R$, we set $E+F = \{x+y\,:\,x\in E, y\in F\}$, and $cF=\{cx\,:\,x\in F\}$.
%Set $\ell = \sqrt{n}/\alpha$.

%%%%%%%%%%%%%%%%%%%%%%%%%%%%%%%%%%%%%%%%%%%

\section{A generalization of the classical Hedlund example} \label{sec:Hed}
In this section, we assume that $n \geq 3$ and we provide a proof of Theorem \ref{thm:polytope}. Our goal is: we construct an $a\in C^\infty(\T^n, (0,\infty))$, such that the effective Hamiltonian $\ol{H}$ determined by ${\rm (E)}_p$ corresponds to the support function of the polytope $P$ given in Theorem \ref{thm:polytope}.

\begin{proof}[Proof of Theorem \ref{thm:polytope}]
Assume that the vertices of $P$ are $\pm q_1, \pm q_2, \ldots, \pm q_m$, which are rational vectors.
Denote by $L_i=\{tq_i\,:\,  t\in  \R \}$ for $1 \leq i \leq m$.
Since $P$ is convex and has nonempty interior, $q_1, q_2,\ldots, q_m$ are mutually non-parallel and 
\[
\text{span} \{q_1, q_2, \ldots, q_m\} = \R^n.
\]
As a result,
\begin{equation*}
\theta :=\min_{|p|=1}\max_{1\leq i\leq m}|p\cdot q_i|>0.
\end{equation*}

\emph{Step 1: Construction of $a$.} Let $x_1=0$. 
Iteratively, for $k\leq m-1$, choose
\[
x_{k+1}\in (0,1)^n\backslash \bigcup_{i=1}^{k}\left\{x_i+sq_{k+1}+tq_i+\Z^n:\ s,t\in \R\right\}.
\]
Then for such selected $x_1,x_2, \ldots, x_m\in (0,1)^n$,  we have that  for $i \neq j$,
\[
(x_i+L_i+\Z^n)\cap (x_j+L_j+\Z^n)=\emptyset.
\]
Due to the fact that $\{q_i\}$ are rational vectors, the projection of $\{x_i + L_i + \Z^n\}$ to the flat torus $\T^n$ is a closed orbit. 
As a result, we can choose a  sufficiently small number $\del \in (0,1/3)$  so that, for $i \neq j$,
\[
T_{\del,i} \cap T_{\del,j} = \emptyset,
\]
where
\[
T_{\del, i} = \{ x\in \R^n\,:\, d(x, x_i + L_i + \Z^n) \leq \del \}.
\]

Choose a smooth $\Z^n$-periodic function $a:\R^n \to (0,\infty)$ such that 
\[
\begin{cases}
a(x)=A|q_i|   \qquad &\text{on $x_i+L_i+\Z^n$ for $1\leq i \leq m$},\\
1\leq a(x)\leq A|q_i| \qquad &\text{on $T_{\delta, i}$ for $1\leq i \leq m$},\\
a(x)=1  \qquad &\text{on $\R^n\backslash \bigcup_{i=1}^{m}T_{\delta, i}$}.
\end{cases}
\]
Here, $A>0$ is a large positive constant to be determined later.
\smallskip

Next, for every unit vector $|p|=1$ and $1\leq i\leq m$, write
\[
p_{i}^{\perp}=p-{(p\cdot q_i)\over |q_i|^2}q_i,
\]
which is the projection of $p$ on the $(n-1)$-dimensional Euclidean space that is perpendicular to $q_i$. 
Apparently, we can construct  a smooth $\Z^n$-periodic function $\phi$  satisfying that 
\[
D\phi(x)=-p_{i}^{\perp} \qquad \text{in $T_{{\delta}, i}$ for all $1\leq i\leq m$},
\]
and
\[
\|D\phi\|_{L^\infty}\leq C_{\delta},
\]
for a constant  $C_{\delta}>0$ depending only on $\delta$, and $q_1, q_2,\ldots, q_m$. 
We now pick $A$ such that
\[
A\geq \max\left\{{1+C_{\delta}\over \theta}, {1\over \min_{1\leq i\leq m}|q_i|}\right\}.
\]

\medskip

\emph{Step 2: Characterization of the effective Hamiltonian}. Let $\ol{H}$ be determined by ${\rm (E)}_p$ with $a$ defined above. We claim that
\begin{equation}\label{H-bar-goal}
\overline H(p)=A\max_{1\leq i\leq m} |q_i\cdot p| \qquad \text{ for all $p\in \R^n$}.
\end{equation}

\smallskip

We only need to prove this claim for unit vectors $|p|=1$. Fix such a $p$. Firstly, by using $\phi$ above, it is clear that
\begin{align*}
\overline H(p)&\leq \max_{x\in \R^n}a(x)|p+D\phi(x)|\\
&\leq \max\left\{A\max_{1\leq i\leq m} |q_i\cdot p|,  1+C_{\delta}\right\}
=A\max_{1\leq i\leq m} |q_i\cdot p|.
\end{align*}
Secondly, let $v=v_p$ be a solution of (E)$_p$.
We assume $v \in C^1(\T^n)$ (to make this rigorous, one needs to do convolution with a standard mollifier, but we omit it here).
 Then, for each $1 \leq i \leq m$,
\[
A|q_i||p+Dv(x)|=\overline H(p)  \quad \text{for $x\in x_i+L_i+\Z^n$}.
\]
Denote by $u(x)=p\cdot x+v(x)$ for $x\in \R^n$.  
Choose $m\in \Z$ such that $mq_i\in \Z^n$. Then
\[
u(x_i+mq_i)-u(x_i)=mp\cdot q_i.
\] 
Since $|u(x_i+mq_i)-u(x_i)|\leq  m|q_i|\max_{x\in x_i+L_i}|Du(x)|$, we deduce that 
\[
\overline H(p)\geq A|p\cdot q_i|.
\]
Accordingly, \eqref{H-bar-goal} holds true.

\medskip

\emph{Step 3: The corresponding shape $D$}. Let $a_A(x)={a(x)\over A}$. Then by scaling the result of Step 2, the effective Hamiltonian $\ol{H}_A(p)$ determined by ${\rm(E)}_p$ with $a$ replaced by $a_A$ satisfies
\begin{equation*}
\ol{H}_A (p) = \max_{1\leq i\leq m} |q_i\cdot p| = \max_{q\in \{\pm q_1, \cdots, \pm q_m\}} q\cdot p = \max_{q \in P} q\cdot p, \qquad \text{ for all $p\in \R^n$}.
\end{equation*}
Compare with Theorem \ref{thm:hom} and in particular the relation \eqref{eq:HD}, we conclude that $P = D_{a_A}$. This completes the proof.
\end{proof}

 \begin{rem}\label{rem:unstability} 
From the constructions in the proof, we observe that, by properly choosing those rational vectors $q_i$ and $\delta$, it is not hard to construct a sequence $\{a_m(\cdot)\} \subset C^{\infty}(\T^n)$ such that 
\[
0<a_m\leq 1, \qquad \lim_{m\to \infty}a_m(x)=0 \quad \text{ for a.e. $x\in \T^n$},
\]
and
\[
\lim_{m \to \infty}\overline H_m(p)=|p| \quad \text{locally uniformly in $\R^n$}.
\]
\end{rem}

%%%%%%%%%%%%%%%%%%%%%%%%%%%%%%%%%%%%%%%%%%%

\section{The optimal rate of convergence} \label{sec:op}
This section is strongly motivated by \cite{MTY}.
It is clear from the proof of Theorem \ref{thm:hom} in the Appendix that, in order to quantify the rate of convergence of $u^\ep$ to $u$, we need a quantitative version of \eqref{K-conv}.
The following result, which is similar to \cite[Lemma 2.1]{MTY}, shows this connection.

\begin{lem}\label{lem:rate-connection}
Let $D$ be the limit shape of Lemma \ref{lem:R-t-conv}. Assume that there exists $C>0$ such that
\begin{equation}\label{R-t-rate}
\rho\left(\frac{\cR_t(Y)}{t}, D \right) \leq \frac{C}{t} \quad \text{ for } t>0.
\end{equation}
Then, there exists $C>0$ such that
\begin{equation}
\label{eq:hom_err}
\|u^\ep-u\|_{L^\infty(\R^n \times [0,\infty))} \leq C \ep.
\end{equation}
\end{lem}

\begin{proof} 
By \eqref{R-t-rate}, 
\[
 \rho \left( \frac{\cR_{t/\ep}(Y)}{t/\ep},D\right)= \rho \left( \ep \frac{\cR_{t/\ep}(Y)}{t},D\right) \leq \frac{C}{t/\ep} = \frac{C\ep}{t},
\]
which implies
\[
 \rho \left( \ep \cR_{t/\ep}(Y),tD\right) \leq C \ep.
\]
We then use this result in \eqref{u-HL} and \eqref{u-ep-rep} to conclude the proof.
\end{proof}

%%%%%%%%%%%%%%%%%%%%%%%%

\subsection{Proof of Theorem \ref{thm:rate2}}
We first give a definition on backward characteristics (see calibrated curves in \cite{Fa}, backward characteristics in \cite[Chapter 5]{Tr}).  

\begin{defn} \label{def:back}
For each $p\in  \R^n$,  let $v_p \in \Lip(\T^n)$ be a viscosity solution to the cell problem {\rm (E)$_p$}.
Then, $\xi: (-\infty, 0]\to \R^n$ is called a backward characteristic of $v_p$ if 
\begin{equation}\label{char}
p\cdot \xi (t_1)+v_p(\xi (t_1))-p\cdot \xi (t_2)-v_p(\xi (t_2))=\int_{t_2}^{t_1}L(\xi (t),\dot \xi(t))+\overline H(p)\,dt
\end{equation}
for all $t_2<t_1\leq 0$. Here $L$ is the Legendre transform of the Hamiltonian $H$.
\end{defn}
\noindent For our problem \eqref{HJ-ep}, the Hamiltonian $H(y,p)=a(y)|p|$ only has linear growth at infinity. The corresponding Lagrangian $L$ is defined by
\begin{equation*}
L(y,v) = \sup_{p \in \R^n} \left( v\cdot p - H(y,p) \right) = \begin{cases}
0 &\qquad |v| \le a(y),\\
+\infty &\qquad |v| > a(y).
\end{cases}
\end{equation*}
By approximating $H$ by smooth and strictly convex Hamiltonians, it is easy to show that for each $p\in \R^n$, 
there exists a viscosity solution $v_p \in \Lip(\T^n)$ of (E)$_p$ such that for every $y\in  \R^n$, 
there is a backward characteristic $\xi_y$ of $v_p$ with $\xi_y(0)=y$. Of course, due to \eqref{char}, $L(\xi_y(s),\dot \xi_y(s))=0$ for all $s<0$.
Hence, $|\dot \xi_y(s)| \leq a(\xi_y(s))$, which yields $\xi_y(s) \in \cR_{|s|}(y)$  for all $s<0$.
In other words, \eqref{char} for $\xi_y$ can be written in an equivalent way as following
\begin{equation}\label{char-n}
\begin{cases}
\xi_y(s) \in \cR_{|s|}(y)  &\text{for } s<0,\\
p\cdot \xi (t_1)+v_p(\xi (t_1))-p\cdot \xi (t_2)-v_p(\xi (t_2))=(t_1-t_2) \ol{H}(p) &\text{for } t_2<t_1 \leq 0.
\end{cases}
\end{equation}

\begin{lem}\label{lem:vertex-rate} 
Assume all the hypotheses of Theorem \ref{thm:rate2}.
Let $v$ be a vertex of $P$.
Pick $p \in \R^n$ such that $\ol{H}$  is differentiable at $p$ and $D\ol{H}(p)=v$.
Let  $\xi :(-\infty, 0]\to \R^n$ be a  backward characteristic of $v_p\in \Lip(\T^n)$, a solution of {\rm (E)}$_p$.
Then, $\xi(t) \in \cR_{|t|}(\xi(0))$ for all $t<0$. Furthermore, there exists $C_p>0$ depending on $H, \ol{H}, p$ such that
\begin{equation}\label{v-rate}
\left |{\xi(t)-\xi(0)\over t}- v\right|\leq {C_p\over |t|}  \quad \text{for $t <0$}.
\end{equation}
\end{lem}

\begin{proof}  
In fact, $\ol{H}$ is linear in a neighborhood of $p$, say $B(p,r)$, and $D^2 \ol{H}(p)=0$.

By the definition of backward characteristics and \eqref{char-n}, $\xi(t) \in \cR_{|t|}(\xi(0))$, and
\[
p\cdot \xi(0)-p\cdot \xi(t)+v_p(\xi (0))-v_p(\xi (t))=|t| \overline H(p),
\]
 for all $t<0$.
Next, for $\tilde p\in  \R^n$, let $v_{\tilde p} \in \Lip(\T^n)$ be a solution to (E)$_{\tilde p}$.  It is clear that 
\[
\tilde p\cdot \xi(0)-\tilde p\cdot \xi(t)+v_{\tilde p}(\xi (0))-v_{\tilde p}(\xi (t))\leq \int_{t}^{0}L( \xi(s),\dot\xi(s) )+\overline H(\tilde p)\,ds= |t| \ol{H}(\tilde p).
\]
It suffices to prove \eqref{v-rate} for $t\leq -1$. Set
\[
w={{\xi (t)-\xi(0)\over t}}-D\overline H(p)={\xi(t)-\xi(0)\over t}- v.
\]
There is nothing to prove if $w=0$. We therefore may assume that $w \neq 0$.
Accordingly, for $|\tilde p-p| < r$, 
\begin{equation}\label{conv-rot-2}
\overline H(\tilde p)-\overline H(p)\geq  {(\tilde p-p)\cdot {\xi (t)-\xi(0)\over t}}-{{C(1+|p|)}\over |t|}.
\end{equation}
Since $\ol H$ is linear in $B(p,r)$, for $|\tilde p-p| <r$, 
\[
\ol{H}(\tilde p)  = \ol{H}(p) + D\ol{H}(p)\cdot (\tilde p - p).
\]
Combine this with the above inequality to deduce that
\[
  (\tilde p-p)\cdot \left( {{\xi (t)-\xi(0)\over t}}-D\overline H(p)\right) \leq {{C(1+|p|)}\over |t|}.
\]
Choose $\tilde p=p+{r \over 2} {w\over  |w|} $ to complete the proof.
\end{proof}

\begin{lem}\label{lem:P-rate1} 
Assume all the hypotheses of Theorem \ref{thm:rate2}.
Then, there exists $C>0$ such that
\[
P \subset \frac{\cR_t(Y)}{t} + B_{\frac{C}{t}}.
\]
\end{lem}

\begin{proof}
Let $v_1,\ldots, v_k$ be the vertices of $P$.
By Lemma \ref{lem:vertex-rate}, there exist $\xi_1(\cdot), \ldots, \xi_k(\cdot)$ and a constant $C>0$ such that
\[
\begin{cases}
\xi_i(t) \in \cR_t(Y) \quad &\text{ for all } 1\leq i \leq k,\, t>0,\\
|\xi_i(t) - t v_i| \leq C \quad &\text{ for all } 1\leq i \leq k,\, t>0.
\end{cases}
\]
For any point $v \in P$, $v$ can be written as a convex combination of the vertices $v_1,\ldots, v_k$, that is,
\[
v = \al_1 v_1 + \cdots + \al_k v_k,
\]
for some $\al_1,\ldots, \al_k \geq 0$ and $\sum_{i=1}^k \al_i=1$.
We construct $\xi(\cdot)$ as a convex combination of the paths $\xi_1(\cdot), \ldots, \xi_k(\cdot)$ (roughly, $\al_1 \xi_1(\cdot)+ \cdots + \al_k \xi_k(\cdot)$ with connectors as in \eqref{eq:connector}) so that
\[
\xi(t) \in \cR_t(Y), \ |\xi(t) - tv| \leq C \quad \text{ for all } t>0.
\]
The proof is complete.
\end{proof}

We now need to obtain the converse inclusion.
It is extremely important noting that the following result always holds true in general setting without any restriction on $P$.
This result is similar to \cite[lower bound (1.2)]{MTY}.

\begin{lem}\label{lem:P-rate2} 
There exists $C>0$ such that
\[
 \frac{\cR_t(Y)}{t} \subset P+ B_{\frac{C}{t}}.
\]
\end{lem}

\begin{proof}
Let $z$ be a point on the boundary of $P$. For given $t>0$, consider $\gam:[0,t] \to \R^n$ such that
\[
\begin{cases}
\gam(0) \in Y, \gam(t) = \lam z,\\
|\dot \gam(s)| \leq a(\gam(s)) \quad \text{ for a.e. } s \in (0,t),
\end{cases}
\]
where $\lam>0$ is given. We now proceed to estimate how big $\lam$ can be.
For each $p\in P$, let $v_p \in \Lip(\T^n)$ be a solution of (E)$_p$. Then,
\[
\ol{H}(p) = a(\gam(s))|p+Dv_p(\gam(s))| \geq \dot \gam(s) \cdot (p + Dv_p(\gam(s))) \quad \text{ for a.e. } s \in (0,t).
\]
To make the above rigorous, one needs to do convolution with a standard mollifier, but we omit it here.
Integrate this on $[0,t]$ to deduce that
\[
t \ol{H}(p) \geq \int_0^t \dot \gam(s) \cdot (p + Dv_p(\gam(s))) \,ds= p\cdot (\lam z  - \gam(0)) + v_p(sz) - v_p(\gam(0)),
\]
which means
\[
p \cdot\frac{\lam z}{t} - \ol{H}(p) \leq \frac{C}{t}.
\]
If $\lam \leq 1$, then $\gam(t) \in P$.
Otherwise, for $\lam>1$, we take the supremum of the above over $p \in P$ to yield
\begin{equation*}
c\left(\frac{\lam}{t}-1\right) \leq \sup_{p\in P} \left( p \cdot \frac{\lam z}{t} - \ol{H}(p)\right) \leq \frac{C}{t},
\end{equation*}
for some $c>0$. Thus, we arrive at
\[
\frac{\lam}{t} \leq 1 + \frac{C}{t},
\]
which gives the conclusion.
\end{proof}

\begin{proof}[Proof of Theorem \ref{thm:rate2}]
Thanks to Lemmas \ref{lem:P-rate1} and \ref{lem:P-rate2}, we infer that
\[
\rho\left(\frac{\cR_t(Y)}{t}, D \right) \leq \frac{C}{t} \quad \text{ for } t>0.
\]
We then apply Lemma \ref{lem:rate-connection} to conclude the proof.
\end{proof}

%%%%%%%%%%%%%%%%%%%%%%%%%%%%%%%%%%%%%%%%%%%
\appendix
%\section*{Appendix}

\section{Homogenization through the shape theorem}
 \label{sec:Rset}

In this appendix, we give quick proofs of Lemma \ref{lem:R-t-conv} and Theorem \ref{thm:hom}, which are classical and standard in the literature. 
The proof of Lemma \ref{lem:R-t-conv} follows closely the ideas in \cite{JST1}.
Once \eqref{R-t-conv} is proved, homogenization of \eqref{HJ-ep} follows in a quite straightforward way. Throughout the proof, $\alpha >0$ and $\beta > \alpha$ denote, respectively, the lower and upper bound of the positive periodic function $a$. 

Let us first note that we can always construct an admissible path $\zeta_{p\to q} \in \cA_{0,\ell}$ to reach from a point $q$ to a point $p$ within time $\ell = \sqrt{d}/\alpha$, if $p$ and $q$ belong to the same cube of unit size, i.e., $p-q \in Y$ or $q-p \in Y$. 
Without loss of generality, assume $p\ne q$ and let $e = (q-p)/|q-p|$ denote the unit directional vector from $p$ to $q$. 
Then the path can be explicitly defined by
\begin{equation}\label{eq:connector}
 \zeta_{p\to q} (t) = \begin{cases}
 p + t\alpha e, & t \in [0,|q-p|/\alpha], \\
 q, & t \in [|q-p|/\alpha,\ell].
 \end{cases}
 \end{equation}

\begin{proof}[Sketch of proof of Lemma \ref{lem:R-t-conv}]
Let $\cR_t(Y)$ be the reachable set at $t > 0$ starting from the unit cell $Y$, that is, $\cR_t(Y)=\bigcup_{x \in Y} \cR_t(x)$. 
Then we have, for all $t,s > 0$,
\begin{equation*}
 \cR_{t+s}(Y) \subset \cR_{s}(Y) + \cR_{t}(Y) + \tilde Y.
 \end{equation*}
By taking closure and then taking the convex hulls, we also get
\begin{equation}
 \co \ol \cR_{t+s}(Y) \subset \co \ol \cR_{s}(Y) + \co \ol \cR_{t}(Y) + \tilde Y.
 \end{equation}
Set $X_m := \co \ol \cR_m(Y) + \tilde Y$ for $m\in \N$. 
We obtain a sequence of convex compact sets $\{X_m\}_{m\in \N}$ which are subadditive in the sense that $X_{m+n} \subset X_m + X_n$.
Then by subadditivity, we get
\begin{equation}
\lim_{n\to \infty} \frac{X_n}{n} = \lim_{n\to \infty} \frac{\co \ol \cR_n(Y)}{n} = D := \bigcap_{n \ge 1}^\infty \frac{X_n}{n}, \qquad \text{in } (\sC, \rho).
\end{equation}
We note that $D$ is a convex and compact set in $\R^n$. Moreover, in view of the controls on $\cR_t$, we have $\ol B_\alpha \subset D \subset \ol B_\beta$. 
By the monotonicity of $\cR_t$ in $t$, and by the fact that $\rho(\co \cR_t, \co \ol \cR_t) = 0$ for all $t$, we also have that
\begin{equation*}
\lim_{t\to \infty} \rho\left(\frac1t \co \cR_t(Y), D\right) = 0.
\end{equation*}

The next step is to show that
\begin{equation*}
\lim_{t\to \infty} \rho\left(\frac1t \cR_t(Y), D\right) = 0
\end{equation*}
still holds, that is to remove the convex hull in the convergence result. Since $\cR_t$ is a subset of its convex hull, we only need to show that
\begin{equation}
\max_{z\in D}\; d\left(z,\frac1t \cR_t(Y)\right) \to 0, \qquad \text{as } t\to \infty.
\end{equation}
If we view the mappings $z \mapsto d(z,t^{-1}\cR_t(Y))$ as a family of functions of $z \in D$ indexed by $t > 0$, then this is an equicontinuous family. 
Because $D$ is a compact set, the above result would follow if we prove that: for every $z \in D$,
\begin{equation}
\label{eq:dlim_pt}
 \lim_{t\to \infty} d\left(z,\frac1t \cR_t(Y)\right) = 0.
 \end{equation}
 Here are the main steps to obtain the above.
First, for exposed points, \eqref{eq:dlim_pt} follows by convex analysis. 
Indeed, if $y$ is an exposed point of $D$, then there exists an affine function $f: \R^n \to \R$ such that $f(y) > f(x)$ for all $x\in D\setminus \{y\}$. 
Let $E_t := \cE(t^{-1}\co \ol \cR_t(Y))$ denote the set of exposed points of the convex hull of $\ol \cR_t(Y)$. 
For each $t>0$, let $x_t$ be a point of $\ol E_t$ where $f$ admits its maximum in $\co \ol \cR_t(Y)$. 
This is possible since $f$ is an affine function and $\co \ol \cR_t(Y)$ is compact and convex, so the maximum can be realized in the closure of the expose set. 
Hence, there exists $x_t \in E_t$ such that
\begin{equation*}
f(x_t/t) \ge \max_{x \in t^{-1}\co \ol \cR_t(Y)} f(x) - \frac1t.
\end{equation*}
Since $\{x_t/t\} \subset \ol B_\beta$, we can find at least one cluster point $z$ of $x_t/t$, and clearly $z \in D$. 
Moreover, passing to the limit in the above equality, we get $f(z) \ge f(y)$.
This shows that the cluster point must be $y$, the pre-fixed exposed point of $D$. 
This shows that
\begin{equation*}
\lim_{t\to \infty} \; d\left(y,\cE \left(\frac1t \co \ol \cR_t(Y)\right)\right)  = 0.
\end{equation*}
On the other hand, since $\cE(\co \ol A) \subset \ol A$, we get
\begin{equation*}
\lim_{t\to \infty} \; d\left(y,\frac1t \ol \cR_t(Y)\right) = 0.
\end{equation*}
Then \eqref{eq:dlim_pt} follows for any $z$ that is an exposed point of $D$.

\smallskip

Next, we consider the case of $z$ being an extreme point of $D$. 
By Straczewicz's theorem, $z$ is a limit point of $\cE(D)$. 
Since \eqref{eq:dlim_pt} hold on $\cE(D)$, it also hold on $\ol \cE(D)$. 
That is it holds for all extreme points of $D$.

\smallskip

Finally, for all other points of $z \in D$, we show $z$ is close to $t^{-1}\cR_t$ as $t\to \infty$ by using the convex combination of extreme points and then by using the fact that extreme points are close to the average reachable set.
More precisely, given any $z \in D$, by Caratheodory's theorem, there exist $n+1$ extreme points $\{y_i\}_i \subset \ol \cE(D)$ and $n+1$ real numbers $\{\lambda_i\}_i \subset [0,1]$ such that $\sum_{i = 1}^{n+1} \lambda_i = 1$, and $z = \sum_{i=1}^{n+1}\lambda_i y_i$.
Without loss of generality, we assume that $\lambda_i > 0$ for all $i$; otherwise we only need to use a smaller number of extreme points.
Fix an arbitrary $\ep > 0$, by \eqref{eq:dlim_pt} for extreme points, there exists a constant $T^\ep_i$, for each $i \in \{1,\cdots, n+1\}$, such that 
\begin{equation}
\label{eq:ind_path}
d\left(y_i, \frac1t \cR_t(Y)\right) < \frac{\ep}{2}, \qquad \text{ for all } t \ge T^\ep_i.
\end{equation}

Next, we show that for all $t > \max\{\sum_{i=1}^{n+1}\lambda_i^{-1}T^\ep_i + n\ell,2\ep^{-1}n\sqrt{n}(1+\beta/\alpha)\}$, we can find a path $\gamma \in \cA_{0,t}$ such that $\gamma(0) \in Y$ and $d(\gamma(t)/t,z) < \ep$. 
This is sufficient to establish \eqref{eq:dlim_pt}.

Let $t' = t - n\ell$. For the first step, we note that $\lambda_1 t' > \lambda_1 T^\ep > T^\ep_1$, so we can find $\gamma_1 \in \cA_{0,\lambda_1 t'}$ such that, if we set $p_1 = \gamma_1(0)$ and $q_1 = \gamma_1(\lambda_1 t')$, then $p_1 \in Y$ and 
\begin{equation*}
\left|y_1 - \frac{q_1}{\lambda_1 t'} \right| < \frac{\ep}{2}.
\end{equation*}
We define the desired path $\gamma$ by $\gamma(s) = \gamma_1(s)$, for $s \in [0,\lambda_1 t']$.
The later steps in the construction of $\gamma$ is by induction. Suppose $\gamma(s)$ is defined for $s \in [0,(k-1)\ell + \sum_{i=1}^k \lambda_i t']$. 
We define the $(k+1)$-th segment of $\gamma$ as follows.
Since $\lambda_{k+1}t' > T^\ep_{k+1}$, we know there exists an admissible path $\gamma_{k+1} \in \cA_{0,\lambda_{k+1}t'}$, and if we set $p_{k+1} = \gamma_{k+1}(0)$ and $q_{k+1} = \gamma_{k+1}(\lambda_{k+1}t')$, then $p_{k+1} \in Y$ and
\begin{equation}
\label{eq:ind_dist}
\left|y_{k+1} - \frac{q_{k+1}}{\lambda_{k+1}t'} \right| < \frac{\ep}{2}.
\end{equation}
Let $T_k = (k-1)\ell + \sum_{i=1}^k \lambda_i t'$, the time upon which $\gamma$ is already constructed. 
Define
\begin{equation*}
\gamma(s) = \begin{cases}
\zeta_{\gamma(T_k) \to [\gamma(T_k)]+p_{k+1}} (s-T_k), &s \in (T_k, T_k+\ell],\\
\gamma_{k+1}(s-T_{k}-\ell), &s \in (T_k+\ell, T_k+\ell + \lambda_{k+1}t'].
\end{cases}
\end{equation*}
The first part of the above construction leads the path $\gamma$ to $p_{k+1} + [\gamma(T_k)]$, which is an integer translation of $p_{k+1}$, and the second part starts from this point and is a translation of $\gamma_{k+1}$. By periodicity of the environment, the translated path is admissible. As a result, we moved the path $\gamma$ forward up to time $T_k + \ell + \lambda_{k+1}t'$.
After a total of $(n+1)$ steps, we have constructed $\gamma$ up to time $\sum_{i=1}^{n+1} \lambda_i t' + n\ell = t$. From the construction, we see that $\gamma \in \cA_{0,t}$. Moreover, we have
\begin{equation*}
 \gamma(t) = \sum_{k=1}^{n+1} q_k + \sum_{k=1}^{n} (\gamma(T_k)-[\gamma(T_k)] - p_k) = \sum_{k=1}^{n+1} q_k + Q
\end{equation*} 
where $Q$ is a point in $nY$. We then have
\begin{equation*}
tz - \gamma(t) = \sum_{k=1}^{n+1} \lambda_k (t'+n\ell) y_k - \sum_{k=1}^{n+1}q_k - Q
= \sum_{k=1}^{n+1} (\lambda_k t' y_k - q_k) - Q + n\ell z.
\end{equation*}
Note that $n\ell z \in n\ell \ol B_{\beta}$. Hence $|Q-n\ell z| \le n\sqrt{n}(1+\beta/\alpha)$. Combining this estimate with \eqref{eq:ind_dist}, we get
\begin{equation*}
\left|z - \frac{\gamma(t)}{t}\right| \le \frac1t \sum_{k=1}^{n+1}|\lambda t' y_k - q_k| + \frac{n\sqrt{n}(1+\beta/\alpha)}{t} < \ep.
\end{equation*}
This completes the proof of \eqref{eq:dlim_pt}.
\end{proof}

We are now ready to give a proof of the qualitative homogenization theorem.
\begin{proof}[Proof of Theorem \ref{thm:hom}]
Note that $x\in \cR_t(y)$ iff $y \in \cR_t(x)$.
Fix $R, T>0$. 
For each $(x,t) \in B_R \times [0,T]$, we use the Hopf-Lax formula for the solution of \eqref{HJ} to imply
\begin{align}\label{u-HL}
u(x,t) =\inf \left \{ g(y)\,:\, \frac{x-y}{t} \in D\right\} &= \inf\left\{g(y)\,:\, y \in x-tD \right\}\\
&= \inf\left\{g(y)\,:\, y \in x+tD \right\}. \notag
\end{align}
On the other hand,
\begin{align}\label{u-ep-rep}
u^\ep(x,t)&= \inf \left \{ g(y)\,:\, \frac{x}{\ep} \in \cR_{\frac{t}{\ep}}\left(\frac{y}{\ep}\right)\right\}=\inf \left \{ g(y)\,:\, \frac{y}{\ep} \in \cR_{\frac{t}{\ep}}\left(\frac{x}{\ep}\right)\right\}\\\
&= \inf\left\{ g(y)\,:\,y \in \ep\left(\left[\frac{x}{\ep}\right] +\cR_{\frac{t}{\ep}}\left(\frac{x}{\ep}-\left[\frac{x}{\ep}\right]\right) \right)\right\} \notag\\
&= \inf\left\{ g(y)\,:\,y \in x+ \ep\left(\cR_{\frac{t}{\ep}}\left(\frac{x}{\ep}-\left[\frac{x}{\ep}\right]\right) -\left(\frac{x}{\ep}-\left[\frac{x}{\ep}\right] \right) \right)\right\}. \notag
\end{align}
By Lemma \ref{lem:R-t-conv}, we yield
\begin{equation}\label{K-conv}
\lim_{\ep \to 0+} \rho \left( \frac{\cR_{t/\ep}(Y)}{t/\ep},D\right)=0,
\end{equation}
which, together with \eqref{u-HL} and \eqref{u-ep-rep}, gives the desired result.
\end{proof}

%%%%%%%%%%%%%%%%%%%%%%%%%%%%%%%%%%%%%%%%%%%

%%%%%%%%%%%%%%%%%%%%%%%%%%%%%%%%%%%%%%%%%%%%

\end{document}